\documentclass[12pt]{iopart}

\usepackage{iopams}

\usepackage{graphicx}
\usepackage{amssymb}
\usepackage{amsthm}
\usepackage{latexsym}
\usepackage{bm}

\usepackage[margin=1in]{geometry}

\theoremstyle{plain}
\newtheorem{theorem}{Theorem}[section]
\newtheorem{proposition}[theorem]{Proposition}
\newtheorem{corollary}[theorem]{Corollary}
\newtheorem{lemma}[theorem]{Lemma}

\newtheorem{assumption}[theorem]{Assumption}
\newtheorem{definition}[theorem]{Definition}

\def\clo#1{\overline{#1}}
\def\text#1{\mbox{#1}}

\def\nullsp#1{\text{{\textit{null}}($#1$)}}
\def\rangesp#1{\text{{\textit{range}}($#1$)}}

\newcommand{\ra}{\rangle}
\newcommand{\la}{\langle}

\newcommand{\Kop}{\mathcal{K}}
\newcommand{\Sph}{\mathbb{S}}

\newcommand{\Gin}{(\partial \Omega \times \Sph)_{-}}
\newcommand{\Gpm}{(\partial \Omega \times \Sph)_{\pm}}

\newcommand{\Ws}{\mathbb{V}^{0}}
\newcommand{\Wss}{\mathbb{V}^{1}}
\newcommand{\Ts}{\mathbb{T}}

\begin{document}

\title[]{Time reversal for radiative transport with applications to inverse and control problems}

\author{Sebastian Acosta}

\address{Computational and Applied Mathematics, Rice University, Houston, TX 77005}
\ead{sebastian.acosta@rice.edu}

\begin{abstract}
In this paper we develop a time reversal method for the radiative transport equation to solve two problems: an inverse problem for the recovery of an initial condition from boundary measurements, and the exact boundary controllability of the transport field with finite steering time. Absorbing and scattering effects, modeled by coefficients with low regularity, are incorporated in the formulation of these problems. This time reversal approach leads to a convergent iterative procedure to reconstruct the initial condition provided that the scattering coefficient is sufficiently small in the $L^{\infty}$ norm. Then, using duality arguments, we show that the solvability of the inverse problem leads to exact controllability of the transport field with minimum-norm control obtained constructively. The solution approach to both of these problems may have applications in areas such as optical imaging and optimization of radiation delivery.
\end{abstract}




\section{Introduction} \label{Section.Intro}

The radiative transport equation, also known as the linear Boltzmann equation, governs the propagation of particles as they interact with the underlying medium. Hence, this equation has applications in various scientific disciplines including optics, astrophysics, nanotechnology, and biology. The mathematical modeling for some of these applications is developed in the following books \cite{Cas-Zwe-1967,Dau-Lio-1993,Agoshkov-1998,Mok-1997,Cer-Gab-2007} among others. This publication is primarily motivated by the connection between transport phenomena and applications such as optical imaging and the optimization of radiation therapy. However, the underlying assumption in this paper is that the medium is \textit{weakly scattering} which may not be verified in some applications. We construct a time reversal method to solve two problems governed by the radiative transport equation: the recovery of an initial condition from boundary measurements, and the exact boundary controllability of the transport field. We shall refer to the former as the \textit{inverse source problem}, and to the latter as the \emph{control problem}.

\subsection{The inverse source problem}

The understanding of transport problems leads to imaging methods for biological media with particular applications in medicine \cite{Wang-Wu-2007,Arr-1999,Gib-Heb-Arr-2005,Kim-Mos-2006,Bal-2009-Rev,Ren-2010}. Another area of interest is the development of wave-based imaging techniques suited for turbid or heavily cluttered media modeled with stochastic differential equations. The connection to the radiative transport equation is provided by the fact that the Wigner transform of stochastic wave fields satisfies the transport equation in specific asymptotic regimes. For details and recent developments in this theory, we refer to \cite{Ryzhik-Papanico-Keller-1996,Bor-Iss-Tso-2010,FGPS-2007,Bal-Pap-Ryz-2002,Gar-Pap-2007,Gar-Sol-2008,Bal-Kom-Ryz-2010,Bal-Ren-2008,Bal-Pinaud-2007,Bal-Pinaud-2011,Luk-Spo-2007}.

The general radiative transport problem is governed by the following system,
\begin{eqnarray*}
& \frac{1}{c} \frac{\partial u}{\partial t} + (\theta \cdot \nabla) u + \mu_{\rm a} u + \mu_{\rm s} (I - \Kop) u = f \quad &\text{in $(0,\infty) \times (\Omega \times \Sph) $},  \\
& u = u_{0} \quad &\text{on $\{ t = 0 \} \times (\Omega \times \Sph)$},
\end{eqnarray*}
augmented by a prescribed in-flow profile which we assume to be vanishing. The properties of the medium are modeled by the absorption coefficient $\mu_{\rm a}$, the scattering coefficient $\mu_{\rm s}$ and the scattering operator $\Kop$. The precise definitions and assumptions concerning these coefficients are made in Section \ref{Section:MainResults}. The transport equation has solutions $u=u(t,x,\theta)$ representing the density of radiation at time $t \in [0,\infty)$, position $x \in \Omega \subset \mathbb{R}^n$, moving in the direction $\theta \in \Sph$ at speed $c > 0$. Here $\Sph$ denotes the unit sphere in $\mathbb{R}^n$. The driving sources in this problem are the actual forcing term in the right-hand side of the equation denoted by $f$ and the initial condition $u_{0}$.

For inverse source problems in transport theory, the goal is to reconstruct some of the driving sources in the formulation of the transport problem from the knowledge of out-flowing boundary data. This theory can be divided into two main areas, namely, transient and stationary problems. To the best of our knowledge, there is no rigorous work concerning the reconstruction of the initial condition $u_{0}$ from boundary measurements for general heterogeneous, absorbing, scattering media. In this paper, we partially fill that void by developing an iterative method, suited for weakly scattering media, to recover the unknown initial condition. See also Section \ref{Section:BeyondWeak} for a brief discussion on how to go beyond the weak scattering assumption.

Most works in the literature are concerned with the stationary case and the recovery of isotropic sources $f=f(x)$. Among those that address uniqueness and stability of reconstructions in scattering media, we highlight \cite{Bal-Tam-2007-01,Ste-Uhl-2008-01,Hub-2011}. We consider our work the counterpart of \cite{Bal-Tam-2007-01} for the time-dependent case. In fact, based on a Neumann series argument, an analogue to that of \cite{Bal-Tam-2007-01}, we have derived an iterative method for the recovery of the unknown initial condition. The method is convergent provided that the scattering coefficient is sufficiently small. This claim is made precise in Section \ref{Section:MainResults} and then proved in Section \ref{Section:InverseProblem}.

For the stationary case, the use of the Neumann series is limited to certain smallness condition for the anisotropic portion of the scattering kernel as shown by Bal and Tamasan \cite{Bal-Tam-2007-01}. Stefanov and Uhlmann \cite{Ste-Uhl-2008-01} by-pass this condition by reducing the inverse source problem to Fredholm form, showing that for generic media an unknown isotropic source can be uniquely recovered in a stable manner from boundary measurements. Their proof assumes full boundary data and is valid for all absorption and scattering coefficients in an open and dense subset of certain normed spaces. The case of partial data has been recently addressed by Hubenthal \cite{Hub-2011} whose approach is an extension of \cite{Ste-Uhl-2008-01}. He is able to show that a source is recoverable if it is supported on certain regions that are \textit{visible} from the accessible portion of the boundary.

We also wish to mention some early works by Larsen \cite{Larsen-1975} and Siewert \cite{Siewert-1993} for special geometries or symmetries. A more recent practical work was carried out by Kim and Moscoso \cite{Kim-Mos-2006} for the problem in half-space with constant coefficients. Using appropriate Green's functions, they developed explicit formulae for the recovery of a point source or a piecewise constant source supported in a box. Most other studies for the stationary inverse source problem are concerned with non-scattering media and make use of the (attenuated) Radon transform and mathematical tools from integral geometry and microlocal analysis. An excellent literature review is found in \cite[Section 7]{Bal-2009-Rev}.

\subsection{The control problem}
Control theory for PDEs is a broad subject which has been investigated by a large number of researchers. We direct the reader to the following publications \cite{Las-Tri-2000,Zua-2007,Lions-1988,Fur-2000,Bar-Leb-Rau-1992,Ell-Mas-2002} some of which contain extended lists of references and overviews of important developments. From these references, it is clear that control theory has been comprehensibly developed for many PDEs of mathematical physics (wave, heat, Maxwell, elasticity, Schrodinger). However, this is not the case for the Boltzmann equation, not even in the linearized case. In fact, the exact controllability for transient radiative transport in heterogeneous media was first established in 2007 by Klibanov and Yamamoto \cite{Kli-Yam-2007}. We develop here an alternative proof for the exact controllability of the transport equation. As in the case of the inverse problem, we are particularly motivated by applications in medicine. For instance, radiation delivery for cancer therapy must be controlled in order to effectively destroy cancerous cells while minimizing damage to surrounding healthy tissues \cite{Shep-Ferr-1999}.

The exact controllability for radiative transport is to find inflow boundary conditions over a window of \textit{steering} time to drive the transport field to a desired final state. The control data is usually not unique, but the first step is to establish its existence. The approach developed in \cite{Kli-Yam-2007} employs Carleman estimates which yield continuous observability even for time-dependent absorption and scattering coefficients, and also for optimal steering time. As a consequence, their results are much stronger than ours. However, we base our work on the analysis of the inverse problem, which is solved using time reversal. Therefore, our work reveals an alternative approach with added value worth reporting in this paper. The precise statement with regard to exact controllability is given in Section \ref{Section:MainResults} and we provide a proof in Section \ref{Section:ControlProblem}.

 
\section{Notation and statement of main results} \label{Section:MainResults}

In this section we state the direct problem for transient radiative transport and its associated time-reversed problem. We also review some preliminary facts in order to state our main results in the proper mathematical framework. The formulation pursued here allows for heterogeneous media modeled by coefficients with low regularity. 

We assume that $\Omega \subset \mathbb{R}^n$ is a bounded convex domain with smooth boundary $\partial \Omega$. The unit sphere in $\mathbb{R}^n$ is denoted by $\Sph$. The respective outflow and inflow portions of the boundary are
\begin{eqnarray*}
\Gpm = \{ (x,\theta) \in \partial \Omega \times \Sph \, : \, \pm \, \nu(x) \cdot \theta > 0 \} 
\end{eqnarray*}
where $\nu$ denotes the outward unit normal vector on $\partial \Omega$. It will be assumed that the particles travel at a fixed speed $c > 0$. In the rest of the paper, we will often make reference to the following scales:
\begin{itemize}
\item[-] $l = \text{\rm diam}(\Omega)$, the diameter of the bounded region $\Omega$, and
\item[-] $T = \text{\rm diam}(\Omega) / c$.
\end{itemize}

Now we define the appropriate Hilbert spaces over which the radiative transport problem is well-posed. First, we denote by $\Ws$ and $\Wss$ the completion of $C^{1}(\clo{\Omega} \times \Sph)$ with respect to the norms associated with the following inner products,
\begin{eqnarray}
&& \la u , w \ra_{\Ws} =  \la u , w \ra_{L^2(\Omega \times \Sph)} \\ \label{Eqn.146}
&& \la u , w \ra_{\Wss} = l^2 \la \theta \cdot \nabla u , \theta \cdot \nabla w \ra_{\Ws} + \la u , w \ra_{\Ws} +  l \la |\nu \cdot \theta| u , w \ra_{L^2(\partial \Omega \times \Sph)}  \label{Eqn.148}
\end{eqnarray}
where $\nabla$ denotes the weak gradient with respect to the spatial variable $x \in \Omega$. Now, denote by $\Ts$ the trace space defined as the completion of $C(\partial \Omega \times \Sph)$ with respect to the norm associated with the following inner product,
\begin{eqnarray}
&& \la u , w \ra_{\Ts} = l \la |\nu \cdot \theta| u , w \ra_{L^2(\partial \Omega \times \Sph)}. \label{Eqn.150}
\end{eqnarray}
We also have the spaces $\Ts_{\pm}$ denoting the restriction of functions in $\Ts$ to the in- and out-flow portions of the boundary $\partial \Omega \times \Sph$, respectively. Functions in $\Wss$ have well-defined traces on the space $\Ts$ as asserted by the following lemma, whose proof is found in \cite{Agoshkov-1998,Mok-1997,Ces-1985,Ces-1984, Egg-Sch-2012-01,Man-Ress-Star-2000-01}. 
 
\begin{lemma} \label{Lemma.101}
The trace mapping $u \mapsto u|_{\partial \Omega}$ defined for $C^{1}(\clo{\Omega} \times \Sph)$ can be extended to a bounded operator $\gamma : \Wss \to \Ts$. Moreover, $\gamma : \Wss \to \Ts$ is surjective. Analogous claims hold for the partial trace maps $\gamma_{\pm} : \Wss \to \Ts_{\pm}$.
\end{lemma}

In addition, we have the following definition for traceless closed subspaces of $\Wss$.
\begin{eqnarray}
\Wss_{\pm} = \nullsp{\gamma_{\pm}} = \left\{ v \in \Wss \, : \, \gamma_{\pm} v = 0 \right\}. \label{Eqn.003}
\end{eqnarray}

One of the most important tools in our analysis is the following integration-by-parts formula or Green's identity. For functions $u,v \in \Wss$ we have
\begin{eqnarray*}
\int_{\Omega \times \Sph} \!\!\!\! (\theta \cdot \nabla u) v  = \int_{\partial \Omega \times \Sph} \!\!\!\! (\theta \cdot \nu) u v  - \int_{\Omega \times \Sph} \!\!\!\! (\theta \cdot \nabla v) u, 
\end{eqnarray*}
and in particular if we let $v = u \in \Wss_{-}$ and multiply by $l$, we obtain from Young's inequality that
\begin{eqnarray}
\frac{l^2}{2} \|(\theta \cdot \nabla) u \|^{2}_{\Ws} + \frac{1}{2} \| u \|^{2}_{\Ws} \geq  \int_{\Omega \times \Sph} \!\!\!\! l (\theta \cdot \nabla u) u  = \frac{l}{2} \int_{\partial \Omega \times \Sph} \!\!\!\! (\theta \cdot \nu) u^2 = \frac{1}{2} \| u \|_{\Ts}^{2}. \nonumber
\end{eqnarray}
Recalling the definition of the norm in $\Wss$ given in (\ref{Eqn.148}), we obtain the following inequality for all $u \in \Wss_{-}$,
\begin{eqnarray}
2 \left( l^2 \|(\theta \cdot \nabla) u \|^{2}_{\Ws} + \| u \|^{2}_{\Ws} \right) \geq \| u \|_{\Wss}^2 . \label{Eqn.092}
\end{eqnarray}

For simplicity, we only pose and analyze the transport problem with vanishing incoming flow. In other words, we shall work in the space $\Wss_{-}$. When a prescribed incoming flow needs to be included, it is easy to lift it as right-hand side source using the surjectivity of the trace operator $\gamma_{-}$ (see lemma \ref{Lemma.101}). The transient radiative transport problem for general heterogeneous, scattering media is the following.

\begin{definition}[Direct Problem] \label{Def.DirectProb}
Given initial condition $u_{0} \in \Wss_{-}$ and forcing term $f \in C^{1}([0,\infty);\Ws) \cup C([0,\infty);\Wss)$, find a solution $u \in C^{1}( [0,\infty); \Ws) \cap C( [0,\infty); \Wss_{-})$ to the following initial boundary value problem
\begin{eqnarray}
\frac{1}{c} \frac{\partial u}{\partial t} + (\theta \cdot \nabla) u + \mu_{\rm a} u + \mu_{\rm s} (I - \Kop) u = \frac{1}{l} f \quad &\text{in $(0,\infty) \times (\Omega \times \Sph) $}, \label{Eqn.005} \\
u = u_{0} \quad &\text{on $\{ t = 0 \} \times (\Omega \times \Sph)$}. \label{Eqn.006}
\end{eqnarray}
\end{definition}

We have included the factor $1/l$ on the right-hand side of (\ref{Eqn.005}) so that both $u$ and $f$ have the same physical units. Here again $\nabla$ denotes the gradient with respect to the spatial variable $x \in \Omega$. Here $\mu_{\rm a}$ and $\mu_{\rm s}$ are the absorption and scattering coefficients, respectively.  The scattering operator $\Kop : \Ws \to \Ws$ is given by
\begin{eqnarray}
 (\Kop u)(x,\theta) = \int_{\Sph} \kappa (x,\theta,\theta') u(x,\theta') \, d S(\theta'), \label{Eqn.009}
\end{eqnarray}
where $\kappa$ is known as the scattering kernel. Throughout the paper we will make the following assumptions concerning the regularity of the absorption and scattering coefficients, and the scattering kernel. 

\begin{assumption} \label{Assump.101}
The scattering coefficient $0 \leq \mu_{\rm s} \in L^{\infty}(\Omega)$. So there exists a positive constant $\clo{\mu}_{\rm s}$ such that $0 \leq \mu_{\rm s}(x) \leq \clo{\mu}_{\rm s}$ for a.a. $x \in \Omega$. Similarly, the absorption coefficient $0 \leq \mu_{\rm a} \in L^{\infty}(\Omega)$, with a constant $\clo{\mu}_{\rm a}$ such that $0 \leq \mu_{\rm a}(x) \leq \clo{\mu}_{\rm a}$ for a.a. $x \in \Omega$.
\end{assumption}

We consider a bounded and conservative scattering operator obtained by making the following assumption concerning the scattering kernel.

\begin{assumption} \label{Assump.105}
The scattering kernel $0 \leq \kappa \in L^{2}(\Omega \times \Sph \times \Sph)$. It is also assumed that the scattering operator is conservative in the following sense,
\begin{eqnarray}
&& \int_{\Sph} \kappa(x,\theta,\theta') d S(\theta') = 1, \quad \text{for a.a. $(x,\theta) \in \Omega \times \Sph$}. \label{Eqn.011}
\end{eqnarray}
In addition, we assume a reciprocity condition on the scattering kernel given by
\begin{eqnarray}
&& \kappa(x,\theta,\theta') = \kappa(x,-\theta' , -\theta ), \quad \text{for a.a. $(x,\theta,\theta') \in \Omega \times \Sph \times \Sph$}. \label{Eqn.013}
\end{eqnarray}
This means that the scattering events are reversible in a local sense at each point $x \in \Omega$.
\end{assumption}

Now we turn our attention to a problem referred to as \textit{reversed transport}. This problem will be used later in the analysis of the inverse problem. It is formally obtained by reversing both time $t$ and direction $\theta$ in the original transport equation (\ref{Eqn.005}), and by employing the reciprocity relation (\ref{Eqn.013}) to obtain the adjoint scattering operator $\Kop^{*}$. This is expressed mathematically in (\ref{Eqn.999}). The reversed transport problem is defined as follows.

\begin{definition}[Reversed Problem] \label{Def.ReverseProb}
Given initial condition $\psi_{0} \in \Wss_{-}$ and forcing term $\rho \in  C^{1}([0,\infty);\Ws) \cup C([0,\infty);\Wss)$, find a solution $\psi \in C^{1}( [0,\infty); \Ws) \cap C( [0,\infty); \Wss_{-})$ to the following initial boundary value problem
\begin{eqnarray}
\frac{1}{c} \frac{\partial \psi}{\partial t} + (\theta \cdot \nabla) \psi - \mu_{\rm a} \psi - \mu_{\rm s} (I - \Kop^{*}) \psi = \frac{1}{l}\rho \quad &\text{in $(0,\infty) \times (\Omega \times \Sph)$}, \label{Eqn.017} \\
\psi = \psi_{0} \quad &\text{on $\{ t = 0 \} \times (\Omega \times \Sph)$}, \label{Eqn.019}
\end{eqnarray}
where the adjoint scattering operator $\Kop^{*} : \Ws \to \Ws$ is given by
\begin{eqnarray}
(\Kop^{*} \psi)(x,\theta) = \int_{\Sph} \kappa(x,\theta',\theta) \psi(x,\theta') \, d S(\theta'). \label{Eqn.020}
\end{eqnarray}
\end{definition}

Before presenting the main results of this paper, we wish to stress the importance of the reciprocity relation (\ref{Eqn.013}). This relation is derived from physical principles of scattering theory usually leading to rotationally invariant kernels of the form $\kappa = \kappa(x,\theta \cdot \theta')$ which satisfy reciprocity. This relation plays a subtle but important role in this paper. It implies that the time reversed equation (\ref{Eqn.017}) (but not boundary or initial conditions) coincides with the so-called \textit{adjoint equation}. As a consequence, we obtain a simple but powerful relationship, expressed mathematically in (\ref{Eqn.1000}), which leads to the equivalence between the solvability of the inverse and control problems. See the proof of theorem \ref{Thm.MainControl} found in Section \ref{Section:ControlProblem}.


\subsection{Main result for the inverse problem} \label{Subsec:MainResultInv}

Now we state the inverse problem for transient transport along with its unique solvability and stability under the assumption that the scattering coefficient is relatively small. Our proof, presented in Section \ref{Section:InverseProblem}, is based on a time reversal method inspired by the work of Stefanov and Uhlmann \cite{Ste-Uhl-2009-01}. Our main goal is to provide a constructive proof that the initial state of the transport field can be uniquely reconstructed from time-resolved boundary measurements. For this inverse source problem, we assume that the properties of the medium are known.

Let $u \in C^{1}( [0,\tau]; \Ws) \cap C( [0,\tau]; \Wss_{-})$ solve the direct transport problem \ref{Def.DirectProb} for unknown initial condition $u_{0} \in \Wss_{-}$. The outflowing boundary measurements are modeled by the operator $\Lambda : \Wss_{-} \to C( [0,\tau]; \Ts_{+})$ defined as
\begin{eqnarray}
(\Lambda u_{0})(t) = \gamma_{+} u(t), \qquad t \in [0,\tau], \label{Eqn.630}
\end{eqnarray}
where $\gamma_{+} : \Wss \to \Ts_{+}$ is the out-flowing trace operator defined in lemma \ref{Lemma.101}. With this notation we define the inverse problem as follows.

\begin{definition}[Inverse Problem] \label{Def.InvProb}
Let $u$ be the solution to the direct problem \ref{Def.DirectProb} for some unknown initial condition $u_{0}$ and forcing term $f = 0$. The inverse source problem is, given the out-flowing measurement $\Lambda u_{0}$, find the initial state $u_{0}$.
\end{definition}

Our main result concerning this inverse problem is the following.

\begin{theorem} \label{Thm.MainInv}
Assume that $l \overline{\mu}_{\rm s} \, e^{l ( \overline{\mu}_{\rm a} + \overline{\mu}_{\rm s} ) }   < e^{-1}$. Then there exists a time $\tau < \infty$ such that the out-flowing boundary measurement $\Lambda u_{0} \in C( [0,\tau]; \Ts_{+})$ determines the initial condition $u_{0} \in \Wss_{-}$ uniquely. Moreover, the following stability estimate
\begin{eqnarray*}
\| u_{0} \|_{\Wss} \leq C \| \Lambda u_{0} \|_{C( [0,\tau]; \Ts_{+})} 
\end{eqnarray*}
holds for some positive constant $C = C(\overline{\mu}_{a}, \overline{\mu}_{s}, l,\tau)$.
\end{theorem}

This theorem is a consequence of theorem \ref{Thm.600} which is stated and proved in Section \ref{Section:InverseProblem}. Notice that theorem \ref{Thm.600} provides a convergent iterative method for the reconstruction of the unknown initial condition $u_{0}$. It is also worth mentioning that the stability estimate of theorem \ref{Thm.MainInv} is optimal for the chosen norms because we can easily show that $\| u_{0} \|_{\Wss}$ dominates $\| \Lambda u_{0} \|_{C( [0,\tau]; \Ts_{+})}$ up to a constant. In other words, in this stability relation we cannot relax the norm on the measured data $\Lambda u_{0}$ without relaxing the norm on the initial condition $u_{0}$ as well. In compliance with this argument, we can prove the validity of theorem \ref{Thm.MainInv} in a weaker or generalized setting, that is, for initial data $u_{0} \in \Ws$. The proof of the following theorem is presented in the remarks after the proof of theorem \ref{Thm.600} in Section Section \ref{Section:InverseProblem}.

\begin{theorem} \label{Thm.MainInv2}
If $l \overline{\mu}_{\rm s} \, e^{l ( \overline{\mu}_{\rm a} + \overline{\mu}_{\rm s} ) }   < e^{-1}$ then there exists $\tau < \infty$ such that
\begin{eqnarray*}
\| u_{0} \|_{\Ws} \leq C \| \Lambda u_{0} \|_{L^{2}( [0,\tau]; \Ts_{+})}, \qquad \text{for all $u_{0} \in \Ws$},
\end{eqnarray*}
for some positive constant $C = C(\overline{\mu}_{a}, \overline{\mu}_{s}, l,\tau)$.
\end{theorem}


\subsection{Main result for the control problem} \label{Subsec:MainResultControl}

Here we proceed to define the control problem and state our main result concerning exact controllability of the transport field from control boundary data. Our proof, based on duality arguments, is presented in Section \ref{Section:ControlProblem}.

For the controllability issue, we work within the framework of a Hilbert space for the control functions on the boundary (generalized traces) and the corresponding mild solutions of the radiative transport problem. For the existence of mild solutions in semigroup theory, see the standard references \cite{Eng-Nag-2000,Pazy-1983}. The treatment of generalized traces for mild solutions can be found in 
\cite[Section 2]{Kli-Yam-2007} or \cite[Section 14.4]{Mok-1997} which is based on Cessenat \cite{Ces-1984,Ces-1985}. See also Bardos \cite[pp. 205-208]{Bar-1970}. 

We consider the following transport problem with prescribed inflow data. Given $h \in L^2([0,\tau];\Ts_{-})$, find a mild solution $v \in  C([0,\tau];\Ws)$ of the following problem
\begin{eqnarray}
\frac{1}{c} \frac{\partial v}{\partial t} + (\theta \cdot \nabla) v + \mu_{\rm a} v + \mu_{\rm s} (I - \Kop) v = 0 \quad &\text{in $(0,\tau] \times (\Omega \times \Sph)$}, \label{Eqn.001c} \\
v = 0 \quad &\text{on $\{ t = 0 \} \times (\Omega \times \Sph)$}, \label{Eqn.002c} \\
v = h \quad &\text{on $[0,\tau] \times (\partial \Omega \times \Sph)_{-}$}. \label{Eqn.003c}
\end{eqnarray}
The inflowing boundary control is modeled by the bounded operator $\Upsilon : L^2([0,\tau];\Ts_{-}) \to \Ws$ defined as
\begin{eqnarray}
\Upsilon h = v(\tau), \label{Eqn.010c}
\end{eqnarray}
where $v$ is the mild solution of the problem (\ref{Eqn.001c})-(\ref{Eqn.003c}). With this notation, we define the control problem in precise terms as follows.

\begin{definition}[Exact Controllability] \label{Def.ControlProb}
Given a target state $v_{\star} \in \Ws$, find a finite steering time $\tau > 0$ and an inflow control condition $h \in L^2([0,\tau];\Ts_{-})$ such that $\Upsilon h = v_{\star}$.
\end{definition}

Our main result concerning this control problem is the following.

\begin{theorem} \label{Thm.MainControl}
Assume that $l \overline{\mu}_{\rm s} \, e^{l ( \overline{\mu}_{\rm a} + \overline{\mu}_{\rm s} ) }   < e^{-1}$. Then there exists a steering time $\tau < \infty$ such that for a given target state $v_{\star} \in \Ws$, there exists inflow control $h \in L^2([0,\tau];\Ts_{-})$ so that the mild solution $v \in C([0,\tau];\Ws)$ of the problem (\ref{Eqn.001c})-(\ref{Eqn.003c}) satisfies $v(\tau) = v_{\star}$.

The control boundary condition has the form $h = h_{\rm min} + g$ where $g \in \nullsp{\Upsilon}$ and $h_{\rm min} \in \nullsp{\Upsilon}^{\perp}$ is uniquely determined by $v_{\star}$ as the minimum-norm control satisfying,
\begin{eqnarray*}
\| h_{\rm min} \|_{L^{2}([0,\tau];\Ts_{-})} \leq C \| v_{\star} \|_{\Ws} 
\end{eqnarray*}
for some positive constant $C = C(\overline{\mu}_{a}, \overline{\mu}_{s}, l, \tau)$.
\end{theorem}


\section{Analysis of the direct problem} \label{Section:TheoryDirect}
 
In this section we briefly review the well-posedness of the direct problem \ref{Def.DirectProb} and the reversed problem \ref{Def.ReverseProb} in appropriate functional spaces. Most of the results in this section are well-known or easily derived from the literature for radiative transport equations. However, we explicitly state the key ideas in order to use them later in the analysis of the inverse and control problems. Our approach is mainly based on semigroup theory and its application to evolution PDEs \cite{Mok-1997,Eng-Nag-2000,Pazy-1983}. Besides reviewing some standard results, the purpose of this section is to specify conditions on the absorbing and scattering coefficients that make the semigroups associated with the two evolution problems decay exponentially fast. See proposition \ref{Prop.010} at the end of this section. It will become clear in the analysis of the inverse problem (see Section \ref{Section:InverseProblem}) that the exponential decay of these semigroups is essential for our proof of the main results of this paper. 

In order to accomplish our goal for this section, we treat the transport problems as bounded perturbations of the ballistic portion of the transport equation. Hence, we start by analyzing the ballistic (non-scattering) transport problems.

\begin{definition}[Ballistic Problems] \label{Def.BallisticProb}
For the direct problem, given initial condition $u_{0} \in \Wss_{-}$, find a solution $u \in C^{1}( [0,\infty); \Ws) \cap C( [0,\infty); \Wss_{-})$ to the following initial boundary value problem
\begin{eqnarray}
\frac{1}{c} \frac{\partial u}{\partial t} + (\theta \cdot \nabla) u +  (\mu_{\rm a} + \mu_{\rm s}) u = 0 \quad &\text{in $(0,\infty) \times (\Omega \times \Sph)$}, \label{Eqn.024} \\
u = u_{0} \quad &\text{on $\{ t = 0 \} \times (\Omega \times \Sph)$}. \label{Eqn.026}
\end{eqnarray}

Similarly, for the time-reversed ballistic problem, given initial condition $\psi_{0} \in \Wss_{-}$, find a solution $\psi \in C^{1}( [0,\infty); \Ws) \cap C( [0,\infty); \Wss_{-})$ to the following initial boundary value problem
\begin{eqnarray}
\frac{1}{c} \frac{\partial \psi}{\partial t} + (\theta \cdot \nabla) \psi - (\mu_{\rm a} + \mu_{s}) \psi = 0 \quad &\text{in $(0,\infty) \times (\Omega \times \Sph)$}, \label{Eqn.025} \\
\psi = \psi_{0} \quad &\text{on $\{ t = 0 \} \times (\Omega \times \Sph)$}. \label{Eqn.027}
\end{eqnarray}
\end{definition}

The ballistic problems are well-posed. In fact, solutions can be written explicitly using the method of characteristics. Since we plan to use the theory of semigroups, we pose these ballistic problems as abstract Cauchy problems in $\Wss_{-}$ with the following operators $A_{0}, B_{0} : \Wss_{-} \to \Ws$ given by
\begin{eqnarray}
A_{0} v = - c \left[ (\theta \cdot \nabla) v + (\mu_{\rm a} + \mu_{\rm s}) v \right],  \label{Eqn.116} \\
B_{0} v = - c \left[ (\theta \cdot \nabla) v - (\mu_{\rm a} + \mu_{\rm s}) v \right],  \label{Eqn.117}
\end{eqnarray}
as their respective generators. We have the following properties concerning the strongly continuous semigroups generated by $A_{0}$ and $B_{0}$.

\begin{lemma} \label{Lemma.140}
Let $\{ S_{0}(t) \}_{t \geq 0}$ and $\{ R_{0}(t) \}_{t \geq 0}$ be the $C_{0}$-semigroups generated by $A_{0}$ and $B_{0}$, respectively. Then the following properties are satisfied.
\begin{itemize}
\item[(i)] Let $v \in \Wss_{-}$ be fixed, then the maps $t \mapsto S_{0}(t)v$ and $t \mapsto R_{0}(t)v$ are continuous from $\mathbb{R}_{+}$ to $\Wss_{-}$ and continuously differentiable from $\mathbb{R}_{+}$ to $\Ws$.
\item[(ii)] For each $t \geq 0$, both $S_{0}(t)$ and $R_{0}(t)$ extend as bounded operators from $\Ws$ to $\Ws$.
\item[(iii)] For all $t \geq 0$, we have $\| S_{0}(t) \| \leq 1$ and $\| R_{0}(t)\| \leq e^{l ( \overline{\mu}_{\rm a} + \overline{\mu}_{\rm s} ) }$.
\item[(iv)] For all $t > T$ we have $S_{0}(t) = R_{0}(t) = 0$.
\item[(v)] For all $t \geq 0$ and all $\omega \leq 0$, we have $\| S_{0}(t) \| \leq e^{\omega (t-T)}$ and $\| R_{0}(t) \| \leq e^{\omega (t-T)} e^{l ( \overline{\mu}_{\rm a} + \overline{\mu}_{\rm s} ) }$.
\end{itemize}
\end{lemma}

\begin{proof}
The proof of properties (i)-(iii) follow from the standard theory of semigroups \cite{Eng-Nag-2000,Pazy-1983} and the fact that the solutions of the Cauchy problems can be explicitly expressed using the method of characteristics. Similarly, property (iv) holds because we impose vanishing inflow condition on the boundary of $\Omega$ by working in the space $\Wss_{-}$. Thus, for time $t > T$ the support of solution $u=u(t)$ or $\psi=\psi(t)$ has already left the domain $\Omega$. Finally, the combination of (iii)-(iv) yields the estimates in (v) as desired.
\end{proof}

Now we turn to the full problem in heterogeneous, absorbing, scattering media. We view both, the direct problem \ref{Def.DirectProb} and the reversed problem \ref{Def.ReverseProb} as perturbations of the ballistic problems \ref{Def.BallisticProb}. So we pose the former two problems as abstract Cauchy problems in $\Wss_{-}$.

The direct problem \ref{Def.DirectProb} corresponds to
\begin{eqnarray}
\dot{u}(t) &= A u(t) + \frac{1}{T} \, f(t) \quad \text{for $t > 0$}, \label{Eqn.122} \\
u(0) &= u_{0},   \label{Eqn.124}
\end{eqnarray}
where $T = l / c$ and the operator $A : \Wss_{-} \subset \Ws \to \Ws$ is given by 
\begin{eqnarray}
A v = - c \left[ (\theta \cdot \nabla) v + \mu_{\rm a} v + \mu_{\rm s} (I-\Kop) v \right].  \label{Eqn.126}
\end{eqnarray}

Similarly, the reversed problem \ref{Def.ReverseProb} corresponds to
\begin{eqnarray}
\dot{\psi}(t) &= B \psi(t) + \frac{1}{T} \, \rho(t) \quad \text{for $t > 0$}, \label{Eqn.132} \\
\psi(0) &= \psi_{0},   \label{Eqn.134}
\end{eqnarray}
where the operator $B : \Wss_{-} \subset \Ws \to \Ws$ is given by 
\begin{eqnarray}
B v = - c \left[ (\theta \cdot \nabla) v - \mu_{\rm a} v - \mu_{\rm s} (I - \Kop^*) v \right].  \label{Eqn.136}
\end{eqnarray}

In order to apply the theory of bounded perturbation of semigroups \cite[Ch. 3]{Eng-Nag-2000}, we find it convenient to state the following lemma, whose proof is a consequence of the Cauchy-Schwarz inequality and assumption \ref{Assump.105} on the scattering kernel $\kappa$.

\begin{lemma} \label{Lemma.145}
Let $\Kop, \Kop^{*}: \Ws \to \Ws$ be defined by (\ref{Eqn.009}) and (\ref{Eqn.020}), respectively. Then we have $\| \Kop \| = \| \Kop^{*} \| = 1$.
\end{lemma}


Now we may state and prove the well-posedness of the direct problem \ref{Def.DirectProb} and reversed problem \ref{Def.ReverseProb} in terms of the associated Cauchy problems (\ref{Eqn.124})-(\ref{Eqn.126}) and (\ref{Eqn.134})-(\ref{Eqn.136}), respectively.

\begin{theorem} \label{Thm.120}
The direct problem \ref{Def.DirectProb} has a unique solution given by the formula
\begin{eqnarray*}
u(t) = S(t)u_{0} + \frac{1}{T} \int_{0}^{t} S(t-s) f(s) ds,
\end{eqnarray*}
where $\{ S(t) \}_{t \geq 0}$ is the $C_{0}$-semigroup generated by $A$ satisfying properties (i)-(ii) of lemma \ref{Lemma.140} and such that $S(t) : \Ws \to \Ws$ is bounded with $\| S(t) \|_{\Ws} \leq  M_{\omega} e^{(\omega + M_{\omega} c \overline{\mu}_{\rm s} )t}$ for all $t \geq 0$ and all $\omega \leq 0$ where $M_{\omega} = e^{- \omega T}$.

Similarly, the reversed problem \ref{Def.ReverseProb} has a unique solution given by the formula
\begin{eqnarray*}
\psi(t) = R(t)\psi_{0} + \frac{1}{T} \int_{0}^{t} R(t-s) \rho(s) ds,
\end{eqnarray*}
where $\{ R(t) \}_{t \geq 0}$ is the $C_{0}$-semigroup generated by $B$ satisfying properties (i)-(ii) of lemma \ref{Lemma.140} and such that $R(t) : \Ws \to \Ws$ is bounded with $\| R(t) \|_{\Ws} \leq N_{\omega} e^{(\omega + N_{\omega} c \overline{\mu}_{\rm s} )t}$ for all $t \geq 0$ and all $\omega \leq 0$ where $N_{\omega} = e^{- \omega T}  e^{l ( \overline{\mu}_{\rm a} + \overline{\mu}_{\rm s} ) }$.
\end{theorem}

\begin{proof}
The proof for both problems \ref{Def.DirectProb} and \ref{Def.ReverseProb} is exactly the same, so we only address the well-posedness of the direct problem \ref{Def.DirectProb}. 
We use \cite[Ch. 3 : Thm. 1.3]{Eng-Nag-2000} and property (v) in lemma \ref{Lemma.140} to obtain the existence of $S(t)$ as a $C_{0}$-semigroup satisfying properties (i)-(ii) of lemma \ref{Lemma.140} and the estimate on its norm $\| S(t) \|_{\Ws}$. Now since we assume that $f \in C([0,\infty);\Wss) \cup C^{1}([0,\infty);\Ws)$, the regularity of the term
\begin{eqnarray*}
\int_{0}^{t} S(t-s) f(s) ds,
\end{eqnarray*}
follows from \cite[Ch. 7: Section 7]{Eng-Nag-2000}. 
\end{proof}

The estimate $\| R(t) \|_{\Ws} \leq N_{\omega} e^{(\omega + N_{\omega} c \overline{\mu}_{\rm s})t}$  found in theorem \ref{Thm.120} motivates us to find conditions on $\overline{\mu}_{\rm a}$, $\overline{\mu}_{\rm s}$ and a good choice of $\omega < 0$ such that the reversed semigroup $R(t)$ becomes a contraction for sufficiently large time $t > 0$. It will be evident that such a property determines the solvability of the inverse problem. See below in Section \ref{Section:InverseProblem}. So consider the function $E : \mathbb{R} \to \mathbb{R}$ given by the exponential rate found in the estimate $\| R(t) \|_{\Ws} \leq N_{\omega} e^{(\omega + N_{\omega} c \overline{\mu}_{\rm s})t}$, or in other words,
\begin{eqnarray*}
E(\omega) = \omega + e^{- \omega T}  e^{l ( \overline{\mu}_{\rm a} + \overline{\mu}_{\rm s} ) } c \overline{\mu}_{\rm s}.
\end{eqnarray*}
We would like to find an optimal choice $\omega^{*} < 0$ that minimizes $E$ over $\mathbb{R}$. Subsequently, we would like to find a condition on $\overline{\mu}_{\rm a}$ and $\overline{\mu}_{\rm s}$ leading to $E(\omega^{*}) < 0$. Clearly, $\omega^{*}$ is the unique solution to $E'(\omega) = 0$ which is easily obtained to be 
\begin{eqnarray*}
\omega^{*} = \frac{\ln \left( T  e^{l ( \overline{\mu}_{\rm a} + \overline{\mu}_{\rm s} ) } c \overline{\mu}_{\rm s}  \right)}{T}.
\end{eqnarray*}
Recalling that $l = Tc$, it follows that
\begin{eqnarray*}
E(\omega^{*}) = \frac{\ln \left( T  e^{l ( \overline{\mu}_{\rm a} + \overline{\mu}_{\rm s} ) } c \overline{\mu}_{\rm s}  \right)}{T} + \frac{1 }{T} = \frac{\ln \left( e^{l ( \overline{\mu}_{\rm a} + \overline{\mu}_{\rm s} ) }  \overline{\mu}_{\rm s} l e \right)}{T}
\end{eqnarray*}
Hence, in order for $E(\omega^{*}) < 0$, we require that $ l \overline{\mu}_{\rm s} e^{l ( \overline{\mu}_{\rm a} + \overline{\mu}_{\rm s} ) }   < e^{-1}$. Now, from the estimates in theorem \ref{Thm.120}, we clearly arrive at the following result.

\begin{proposition} \label{Prop.010}
Let $l = \text{\rm diam}(\Omega)$. If 
\begin{eqnarray*}
l \overline{\mu}_{\rm s} \, e^{l ( \overline{\mu}_{\rm a} + \overline{\mu}_{\rm s} ) }   < e^{-1}
\end{eqnarray*}
then both $\| S(t) \|_{\Ws}$ and $\| R(t) \|_{\Ws}$ decay exponentially fast according to
\begin{eqnarray*}
\| S(t) \|_{\Ws}  \leq \| R(t) \|_{\Ws} & \leq e \, e^{ l ( \overline{\mu}_{\rm a} + \overline{\mu}_{\rm s} )}  \left( e  \, e^{l ( \overline{\mu}_{\rm a} + \overline{\mu}_{\rm s} ) }  \, l \overline{\mu}_{\rm s}  \right)^{t/T - 1}   , \qquad t \geq T.
\end{eqnarray*}
\end{proposition}


\section{Analysis of the stationary problem} \label{Section:StationProblem}

In this section we state the stationary or steady-state problem for the transport equation in general heterogeneous, absorbing, scattering media. This is done for the stationary problems corresponding to both, the \textit{direct problem} \ref{Def.DirectProb} and also the \textit{reversed problem} \ref{Def.ReverseProb}. We shall prove the well-posedness of both problems under the same assumptions of proposition \ref{Prop.010}. 

In the analysis of the inverse problem, it will become clear that the well-posedness of the reversed stationary problem plays two important roles. The first of these roles has to do with the definition of a time reversal operator that solves the inverse problem up to a contraction map. More precisely, this operator is written in terms of a time-reversed evolution problem which in turn needs an initial state to be well-defined. The needed initial state is then provided by the solution of a reversed stationary problem in order to obtain a crucial stability property. The second role is that a carefully chosen solution for a reversed stationary problem will allow us to stay within the framework of the space $\Wss$. As a consequence, the stability of the source reconstruction method is given in terms of the $\Wss$-norm which is stronger than that of the space $\Ws$. 

The reversed problem is particularly challenging since the coefficients in the PDE appear with the reversed signs. This means that the term containing $\mu_{\rm a}$ acts as emission instead of absorption. As a result, the weak formulation of this problem does not lead to a coercive (uniformly convex) form. Instead, it leads to a problem that is not necessarily positive definite. We shall overcome this difficulty by employing the celebrated theorem of Babu\v{s}ka \cite{Babuska-1971-01}. This is a generalization of the Lax-Milgram lemma originally designed for similar saddle-point problems. Our approach follows the ideas developed in \cite{Egg-Sch-2012-01,Man-Ress-Star-2000-01}, but we introduce certain simplifications and modifications. 

Now we proceed to state the direct stationary problem with vanishing incoming flow. In other words, we shall work in the space $\Wss_{-}$. As in the case of transient transport, when a prescribed incoming flow needs to be included, it can be lifted using the surjectivity of the trace operator $\gamma_{-}$ (see lemma \ref{Lemma.101}).

\begin{definition}[Direct Stationary Problem] \label{Def.SteadyDirectProb}
Given  $f \in \Ws$ find a solution $u \in \Wss_{-}$ to the following boundary value problem
\begin{eqnarray}
(\theta \cdot \nabla) u + \mu_{\rm a} u + \mu_{\rm s} (I - \Kop) u =  \frac{1}{l} f \quad \text{in $\Omega \times \Sph$}. \label{Eqn.085}
\end{eqnarray}
\end{definition}

And the reversed stationary problem is defined as follows.

\begin{definition}[Reversed Stationary Problem] \label{Def.SteadyReverseProb}
Given $\rho \in \Ws$ find a solution $\psi \in \Wss_{-}$ to the following boundary value problem
\begin{eqnarray}
(\theta \cdot \nabla) \psi - \mu_{\rm a} \psi - \mu_{\rm s} (I - \Kop^{*}) \psi = \frac{1}{l} \rho \quad \text{in $\Omega \times \Sph$}. \label{Eqn.087}
\end{eqnarray}
\end{definition}

We shall use stability estimates for the ballistic portion of the stationary problems. These may be obtained from Poincar\'{e}-type inequality to estimate the $\Ws$-norm of a function in terms of its directional derivative provided that inflow values vanish. Since the ballistic solution can be written explicitly using an integrating factor, the proof of a Poincar\'{e} inequality from \cite{Man-Ress-Star-2000-01}
can be easily modified to obtain the following stability estimates. 

\begin{lemma} \label{Lemma.105}
Let $l = \text{\rm diam}(\Omega)$. For all $u \in \Wss_{-}$ we have that
\begin{eqnarray*}
\| u \|_{\Ws} \leq 2^{-1/2}  \, l  \, \| (\theta \cdot \nabla) u + (\mu_{\rm a} + \mu_{\rm s}) u  \|_{\Ws} \\
\| (\theta \cdot \nabla) u \|_{\Ws} \leq  \left( 1 + 2^{-1/2} \, l \, (\overline{\mu}_{\rm a} + \overline{\mu}_{\rm s}) \right)  \, \| (\theta \cdot \nabla) u  + (\mu_{\rm a} + \mu_{\rm s}) u  \|_{\Ws}.
\end{eqnarray*}
Similarly, for all $\psi \in \Wss_{-}$ we also have
\begin{eqnarray*}
\| \psi \|_{\Ws} \leq 2^{-1/2} \, l \, e^{l (\overline{\mu}_{\rm a} + \overline{\mu}_{\rm s})}  \, \| (\theta \cdot \nabla) \psi  - (\mu_{\rm a} + \mu_{\rm s}) \psi  \|_{\Ws} \\
\| (\theta \cdot \nabla) \psi \|_{\Ws} \leq  \left( 1 + 2^{-1/2} \, l \, (\overline{\mu}_{\rm a} + \overline{\mu}_{\rm s}) \, e^{l (\overline{\mu}_{\rm a} + \overline{\mu}_{\rm s})} \right)  \, \| (\theta \cdot \nabla) \psi  - (\mu_{\rm a} + \mu_{\rm s}) \psi  \|_{\Ws} .
\end{eqnarray*}
\end{lemma}

The constants appearing in the above lemma will apear several times in the sequel. Hence, in order to shorten some expressions, we introduce the following notation,
\begin{eqnarray}
\alpha_{0} = 1 + \sqrt{2} + l \, (\overline{\mu}_{\rm a} + \overline{\mu}_{\rm s}),  \qquad 
\beta_{0} = \sqrt{2} + \left( 1 +  l \, (\overline{\mu}_{\rm a} + \overline{\mu}_{\rm s}) \right) \, e^{l (\overline{\mu}_{\rm a} + \overline{\mu}_{\rm s})}. \label{Eqn.1002}
\end{eqnarray}
As a consequence of lemma \ref{Lemma.105} and (\ref{Eqn.092}), we obtain the following estimates concerning the ballistic operators,
\begin{eqnarray}
\| u \|_{\Wss} \leq  l \, \alpha_{0} \, \| (\theta \cdot \nabla) u + (\mu_{\rm a} + \mu_{\rm s}) u \|_{\Ws} \qquad \text{for all $u \in \Wss_{-}$}, \label{Eqn.1005} \\
\| \psi \|_{\Wss} \leq  l \, \beta_{0} \, \| (\theta \cdot \nabla) \psi - (\mu_{\rm a} + \mu_{\rm s}) \psi \|_{\Ws} \qquad \text{for all $\psi \in \Wss_{-}$}. \label{Eqn.1007}
\end{eqnarray}

In order to prove the well-posedness of the stationary problems \ref{Def.SteadyDirectProb} and \ref{Def.SteadyReverseProb}, we first set up associated variational problems similar to those of \cite{Egg-Sch-2012-01}. However, as opposed to \cite{Egg-Sch-2012-01}, we directly seek for a solution in the trial space $\Wss_{-}$ which already has enough regularity for the solution to satisfy the transport PDE in a strong sense. Without further ado, we define the bilinear forms governing the stationary problems \ref{Def.SteadyDirectProb}-\ref{Def.SteadyReverseProb}.

Let $a , b : \Wss_{-} \times \Ws \to \mathbb{R}$ be bilinear forms given by
\begin{eqnarray}
a(u,v) &= l \, \la (\theta \cdot \nabla) u + \mu_{\rm a} u + \mu_{\rm s} (I - \Kop) u , v \ra_{\Ws}, \label{Eqn.101} \\
b(\psi, \phi) &= l \, \la (\theta \cdot \nabla) \psi - \mu_{\rm a} \psi - \mu_{\rm s} (I - \Kop^{*}) \psi , \phi \ra_{\Ws}. \label{Eqn.102}
\end{eqnarray}

Hence, the stationary problems \ref{Def.SteadyDirectProb}-\ref{Def.SteadyReverseProb} are equivalent to find $u, \psi \in \Wss_{-}$ such that 
\begin{eqnarray}
a(u,v) &= \la f , v \ra_{\Ws} \quad \text{for all $v \in \Ws$}, \label{Eqn.105} \\
b(\psi, \phi) &= \la \rho , \phi \ra_{\Ws} \quad \text{for all $\phi \in \Ws$}. \label{Eqn.107}
\end{eqnarray}

As stated above, the reversed problem \ref{Def.SteadyReverseProb}, governed by the bilinear form $b(\cdot,\cdot)$ poses the greater challenge due to its lack of positive definiteness. Hence, we focus on this case, and the steps can be easily modified to deal with the more favorable structure of the other bilinear form $a(\cdot,\cdot)$. We start with the following lemma.

\begin{lemma} \label{Lemma.209}
If $ l \overline{\mu}_{s} \, e^{l(\overline{\mu}_{a} + \overline{\mu}_{s})} < e^{-1}$, then for each $\psi \in \Wss_{-}$ there exists $\phi \in \Ws$ such that $b(\psi,\phi) \geq \beta  \| \psi \|_{\Wss} \| \phi \|_{\Ws}$, where the constant $\beta > 0$ is independent of $\psi$ and $\phi$. In fact, 
\begin{eqnarray*}
\beta = \frac{ \left( \sqrt{2} \, e  - 1 \right) }{ \sqrt{2} \, e \, \beta_{0}} ,
\end{eqnarray*}
where $\beta_{0}$ is defined in (\ref{Eqn.1002}). Moreover, for each non-zero $\phi \in \Ws$ there exists $\psi \in \Wss_{-}$ such that $b(\psi,\phi) > 0$.

An analogous claim holds true for the other bilinear form $a(\cdot,\cdot)$ except that the constant $\beta$ is replaced by
\begin{eqnarray*}
\alpha = \frac{\left( \sqrt{2} \, e - 1 \right) }{ \sqrt{2} \, e \, \alpha_{0}}
\end{eqnarray*}
where $\alpha_{0}$ is defined in (\ref{Eqn.1002}).
\end{lemma}

\begin{proof}
For given $\psi \in \Wss_{-}$, let $ \phi = l \, \left( (\theta \cdot \nabla) \psi - (\mu_{\rm a} + \mu_{\rm s}) \psi \right) \in \Ws$. Then we have,
\begin{eqnarray*}
 b(\psi, \phi) &= \| \phi \|^{2}_{\Ws} + l \la \mu_{\rm s} \Kop^{*} \psi , \phi \ra_{\Ws} \\
&\geq \| \phi \|^{2}_{\Ws} - l \overline{\mu}_{s} \, \| \psi \|_{\Ws} \| \phi \|_{\Ws} \\
&\geq  \left( 1 - 2^{-1/2} \, l \overline{\mu}_{s} \, e^{l(\overline{\mu}_{a} + \overline{\mu}_{s})}  \right) \| \phi \|^{2}_{\Ws} \geq  \left( 1 - (\sqrt{2} \, e )^{-1} \right) \| \phi \|^{2}_{\Ws}
\end{eqnarray*}
where we have used an estimate from lemma \ref{Lemma.105}. Now it only remains to show that $\| \phi \|_{\Ws}$ dominates $\| \psi \|_{\Wss}$, which is precisely what we obtain from (\ref{Eqn.1007}).

For the second part, given non-zero $\phi \in \Ws$, we choose $\psi \in \Wss_{-}$ such that $ (\theta \cdot \nabla) \psi - (\mu_{\rm a} + \mu_{\rm s}) = \phi / l$ in $(\Omega \times \Sph)$. This can indeed be accomplished since the ballistic stationary problem with vanishing inflow data is uniquely solvable in $\Wss_{-}$. Therefore, we have
\begin{eqnarray*}
b(\psi, \phi) &\geq  \| \phi \|^{2}_{\Ws} - l \overline{\mu}_{s} \| \psi \|_{\Ws} \| \phi \|_{\Ws} \\
&\geq  \| \phi \|^{2}_{\Ws} \left( 1 - 2^{-1/2} \, l \overline{\mu}_{\rm s} \,  e^{l (\overline{\mu}_{\rm a} + \overline{\mu}_{\rm s})} \right)   \\
&\geq  \| \phi \|^{2}_{\Ws} \left( 1 - (\sqrt{2} \, e )^{-1} \right) > 0,
\end{eqnarray*}
where we have used an estimate from lemma \ref{Lemma.105}. The proof for the other bilinear form $a(\cdot,\cdot)$ is analogous using the estimate (\ref{Eqn.1005}).
\end{proof}

The above lemma holds under the condition that $ l \overline{\mu}_{s} \, e^{l(\overline{\mu}_{a} + \overline{\mu}_{s})} < e^{-1}$. This choice was purposely made to coincide with the hypothesis of proposition \ref{Prop.010}. It will become clear in Section \ref{Section:InverseProblem} that the conclusion of proposition \ref{Prop.010} is crucial for the well-posedness of the inverse problem. Now we can state and prove the well-posedness of the stationary problems. Lemma 4.4 and Babu\v{s}ka's theorem \cite{Babuska-1971-01} lead to the well-posedness of the stationary problems as stated in the following theorem.

\begin{theorem} \label{Thm.550}
If $ l \overline{\mu}_{s} \, e^{l(\overline{\mu}_{a} + \overline{\mu}_{s})} < e^{-1}$, then for each $\rho \in \Ws$, then the direct stationary problem \ref{Def.SteadyDirectProb} and the reversed stationary problem \ref{Def.SteadyReverseProb} are well-posed with stability estimate of the form,
\begin{eqnarray*}
\| u \|_{\Wss} \leq \frac{1}{\alpha} \| f \|_{\Ws}. \qquad \text{and} \qquad \| \psi \|_{\Wss} \leq \frac{1}{\beta} \| \rho \|_{\Ws},
\end{eqnarray*}
where the constant $\alpha, \beta > 0$ are defined in lemma \ref{Lemma.209}.
\end{theorem}

We also wish to state here the following corollary that we will employ in the analysis of the inverse problem in Section \ref{Section:InverseProblem}.

\begin{corollary} \label{Cor.001}
Let $\{ S(t) \}_{t \geq 0}$ and $\{ R(t) \}_{t \geq 0}$ be the $C_{0}$-semigroups associated with the direct and reversed problems \ref{Def.DirectProb} and \ref{Def.ReverseProb}, respectively. If $ l \overline{\mu}_{s} \, e^{l(\overline{\mu}_{a} + \overline{\mu}_{s})} < e^{-1}$, then these semigroups are bounded linear operators from $\Wss_{-}$ to $\Wss_{-}$ satisfying the following estimates
\begin{eqnarray*}
\fl \| S(t) \|_{\Wss} & \leq \frac{ 1 + l (\overline{\mu}_{a} + 2 \overline{\mu}_{s} ) }{\alpha } \, \| S(t) \|_{\Ws}  \qquad \text{and} \qquad   
\| R(t) \|_{\Wss} & \leq \frac{ 1 + l (\overline{\mu}_{a} + 2 \overline{\mu}_{s})  }{\beta } \, \| R(t) \|_{\Ws} 
\end{eqnarray*}
where $\alpha, \beta > 0$ are defined in lemma \ref{Lemma.209}.
\end{corollary}

\begin{proof}
Let $u_{0} \in \Wss_{-}$ be arbitrary and in view of theorem \ref{Thm.550}, we have
\begin{eqnarray*}
\| S(t) u_{0} \|_{\Wss} & \leq \frac{l}{c \alpha} \| \dot{u}(t) \|_{\Ws}  = \frac{l}{c \alpha} \| A S(t) u_{0}  \|_{\Ws} = \frac{1}{c \alpha} \| S(t) A  u_{0}  \|_{\Ws} \\
& \leq \frac{l}{c \alpha} \| S(t) \|_{\Ws} \| A  u_{0}  \|_{\Ws} \leq \frac{l}{c \alpha} \| S(t) \|_{\Ws} \| A \|  \| u_{0}  \|_{\Wss}.
\end{eqnarray*}
We have used the fact that a semigroup commutes with its generator when acting on the domain of the generator. In other words, $S(t)A v = AS(t) v $ for all $v \in \Wss_{-}$. The boundedness of $S(t) : \Ws \to \Ws$ was obtained in theorem \ref{Thm.120}, and it is clear that $A : \Wss_{-} \to \Ws$ is bounded with $\|A\| \leq c ( l^{-1} + \overline{\mu}_{a} + 2 \overline{\mu}_{s})$. The proof for $R(t)$ is similar.
\end{proof}


\section{The inverse problem} \label{Section:InverseProblem}

With the tools developed in the previous two sections, we are ready to solve the inverse problem \ref{Def.InvProb}. We base our analysis on the time reversal method, which is usually employed for the wave equation. In particular, our approach mimics that of \cite{Ste-Uhl-2009-01} for the thermoacoustic tomography problem. The time reversal method is particularly well-suited for the wave equation because that equation is invariant under time inversion. Unfortunately, the radiative transport equation does not enjoy such a property, and this represents our main challenge. However, the theory expanded in Sections \ref{Section:TheoryDirect}-\ref{Section:StationProblem} provides the necessary tools to deal with this difficulty.

Let $u \in C^{1}( [0,\tau]; \Ws) \cap C( [0,\tau]; \Wss_{-})$ solve the Cauchy problem,
\begin{eqnarray}
\dot{u}(t) &= A u(t) \quad \text{for $0 < t \leq \tau$}, \label{Eqn.622} \\
u(0) &= u_{0},   \label{Eqn.624}
\end{eqnarray}
where the operator $A : \Wss_{-} \to \Ws$ is defined in (\ref{Eqn.126}) and the initial condition $u_{0} \in \Wss_{-}$. The boundary measurements are modeled by the operator $\Lambda : \Wss_{-} \to C( [0,\tau]; \Ts_{+})$ defined by (\ref{Eqn.630}) for some chosen time $\tau > 0$. Recall also Definition \ref{Def.InvProb} of the inverse problem.

We wish to determine the existence of a finite measurement time $\tau$ which ensures the recovery of $u_{0}$ from $\Lambda u_{0}$. Clearly we need $\tau \geq T = \text{diam}(\Omega) / c$ because $T$ is the time required for the special case when the medium is non-scattering.

In order to employ a time reversal argument, we define the following \textit{reflection operator} $U : C( [0,\tau]; \Wss) \to C( [0,\tau]; \Wss)$ given by
\begin{eqnarray}
(U v)(x,\theta , t) = v(x, - \theta , \tau - t). \label{Eqn.633}
\end{eqnarray}
This operator is certainly bounded, and in fact it is norm-preserving. Moreover, it is also well-defined and norm-preserving as $U : C^{1}( [0,\tau]; \Ws) \to C^{1}( [0,\tau]; \Ws)$ and $U : C( [0,\tau]; \Ts_{\pm}) \to C( [0,\tau]; \Ts_{\mp})$, and in all cases $U^2 = I$. We will also use the fact that 
\begin{eqnarray}
U \Kop = \Kop^{*} U \label{Eqn.999}
\end{eqnarray} 
which follows from the reciprocity condition (\ref{Eqn.013}) satisfied by the scattering kernel $\kappa$.

It is not hard to see that $\tilde{u} := (Uu) \in C^{1}( [0,\tau]; \Ws) \cap C( [0,\tau]; \Wss)$ solves the following reversed initial boundary value problem,
\begin{eqnarray}
\dot{\tilde{u}}(t) = B \tilde{u}(t) \quad \text{for $0 \leq t \leq \tau$}, \label{Eqn.640} \\
\tilde{u}(0) = (Uu)(0),   \label{Eqn.642} \\
\gamma_{-} \tilde{u}(t) = ( U \Lambda u_{0})(t) \quad \text{for $0 \leq t \leq \tau$},  \label{Eqn.644}
\end{eqnarray}
where the generator $B : \Wss_{-} \to \Ws$ is defined in (\ref{Eqn.136}).

If we had access to $(Uu)(0)$, then we could solve (\ref{Eqn.640})-(\ref{Eqn.644}) and it would follow that $u_{0} = (U \tilde{u})(0)$. Unfortunately, this is not realistic and we only have access to the boundary measurements $\Lambda u_{0}$. For the time reversal method, $(Uu)(0)$ is simply replaced by a known function $\psi$ of our choice. 

Inspired by the work of Stefanov and Uhlmann \cite{Ste-Uhl-2009-01} we employ the reversed stationary problem \ref{Def.SteadyReverseProb} as a lift to obtain an initial condition $\psi$ that conforms to the boundary data $(U \Lambda u_{0})(t)$ at $t=0$. Such a function would satisfy a boundary value problem of the form,
\begin{eqnarray}
(\theta \cdot \nabla) \psi - \mu_{\rm a} \psi - \mu_{\rm s} (I - \Kop^{*}) \psi = 0 \quad \text{in $(\Omega \times \Sph)$}, \label{Eqn.701} \\
\gamma_{-} \psi = h_{0} \quad \text{on $\Gin$}. \label{Eqn.702}
\end{eqnarray}
for $h_{0} \in \Ts_{-}$. As seen below, in practice we choose $h_{0} = (U \Lambda u_{0})(0)$.

With this choice of $\psi$, we proceed to define the time reversal operator which acts as an approximate left-inverse for $\Lambda$. Given $h \in C([0,\tau];\Ts_{+})$, find $v \in C^{1}( [0,\tau]; \Ws) \cap C( [0,\tau]; \Wss)$ satisfying
\begin{eqnarray}
\dot{v}(t) = B v(t) \quad \text{for $0 \leq t \leq \tau$}, \label{Eqn.750} \\
v(0) = \psi,   \label{Eqn.752} \\
\gamma_{-} v(t) = (U h)(t) \quad \text{for $0 \leq t \leq \tau$}.  \label{Eqn.754}
\end{eqnarray}
where $\psi$ satisfies (\ref{Eqn.701})-(\ref{Eqn.702}) with $h_{0} = (Uh)(0)$.

The time reversal operator $G : C([0,\tau];\Ts_{+}) \to \Wss$ is given by
\begin{eqnarray}
G h = (U v)(0). \label{Eqn.756}
\end{eqnarray}

We should prove the well-posedness of the boundary value problem (\ref{Eqn.701})-(\ref{Eqn.702}) and the initial boundary value problem (\ref{Eqn.750})-(\ref{Eqn.754}). These problems are not the same as the respective reversed problems \ref{Def.SteadyReverseProb} and \ref{Def.ReverseProb} because now we are prescribing a non-zero inflow data in (\ref{Eqn.702}) and (\ref{Eqn.754}). However, since the trace operator $\gamma_{-}$ is surjective, we can always lift the boundary data and pose new problems in the $\Wss_{-}$. Then, theorems \ref{Thm.550} and \ref{Thm.120} yield the well-posedness of these two problems, respectively. As a consequence, then $G$ is a bounded operator. 

We are interested in making $I - G \Lambda$ a contraction mapping in the space $\Wss_{-}$ for some properly chosen measurement time $\tau < \infty$. This is the basic idea employed in \cite{Ste-Uhl-2009-01}. We proceed to state our result in the form of a theorem.

\begin{theorem} \label{Thm.600}
Assume that $l \overline{\mu}_{s} \, e^{l(\overline{\mu}_{a} + \overline{\mu}_{s})} < e^{-1}$. Let $Q : \Wss_{-} \to \Wss_{-}$ be given by $Q = I - G \Lambda$. There exists a final time $\tau < \infty$ such that $Q: \Wss_{-} \to \Wss_{-}$ is a contraction and $(I - Q) : \Wss_{-} \to \Wss_{-}$ is boundedly invertible. Moreover, the solution to the inverse problem \ref{Def.InvProb} is given by
\begin{eqnarray*}
u_{0} = \sum_{n=0}^{\infty} Q^n G h, \qquad h= \Lambda u_{0}, 
\end{eqnarray*}
with convergence in the $\Wss$-norm.
\end{theorem}

\begin{proof}
Following the approach in \cite[Thm. 1]{Ste-Uhl-2009-01}, let $w \in C^{1}( [0,\tau]; \Ws) \cap C([0,\tau]; \Wss_{-}) $ solve the following Cauchy problem,
\begin{eqnarray}
\dot{w}(t) &= B w(t) \quad \text{for $0 \leq t \leq \tau$}, \label{Eqn.770} \\
w(0) &= (Uu)(0) - \psi \in \Wss_{-}.   \label{Eqn.772}
\end{eqnarray}
Notice that this is in the precise form of the reversed Cauchy problem \ref{Def.ReverseProb} and that the initial condition in (\ref{Eqn.772}) truly belongs to $\Wss_{-}$ by design. Hence, theorem \ref{Thm.120} implies that $w(t) = R(t)((Uu)(0)-\psi)$. Since (\ref{Eqn.770}) is satisfied in a strong sense at $t=\tau$, then it follows from theorem \ref{Thm.550} that
\begin{eqnarray}
\| w(\tau) \|_{\Wss} \leq \frac{l}{c \beta} \| \dot{w}(\tau) \|_{\Ws}.  \label{Eqn.775}
\end{eqnarray}

Notice also that $(v + w)$ and $(Uu)$ solve the same initial boundary value problem (\ref{Eqn.640})-(\ref{Eqn.644}). By uniqueness then $w = Uu - v$ for all $t \in [0,\tau]$, in particular if we apply $U$ and evaluate at $t=0$ we obtain 
\begin{eqnarray}
(Uw)(0) = u_{0} - G \Lambda u_{0} = Q u_{0}, \nonumber 
\end{eqnarray}
where we have used the definition of $G$ given in (\ref{Eqn.756}). Therefore, the following estimates hold
\begin{eqnarray*}
\| Q u_{0} \|_{\Wss} &= \| (U w)(0) \|_{\Wss} =  \| w(\tau) \|_{\Wss} \leq \frac{l}{c \beta} \| \dot{w}(\tau) \|_{\Ws} =  \frac{l}{c \beta} \| B w(\tau) \|_{\Ws} \\
& = \frac{l}{c \beta} \| B R(\tau) ((U u)(0) - \psi)  \|_{\Ws} = \frac{l}{c \beta} \| R(\tau) B ((U u)(0) - \psi)  \|_{\Ws} \\
& = \frac{l}{c \beta} \| R(\tau) B (U u)(0) \|_{\Ws} \leq \frac{l}{c \beta} \| R(\tau) \|_{\Ws} \| B \|  \| (Uu)(0) \|_{\Wss} \\
& = \frac{l}{c \beta} \| R(\tau) \|_{\Ws} \|B \|  \| u(\tau) \|_{\Wss} = \frac{l}{c \beta} \| R(\tau) \|_{\Ws} \| B \|  \| S(\tau) u_{0} \|_{\Wss} \\
& \leq \frac{l}{c \beta} \| R(\tau) \|_{\Ws} \| B \| \frac{l}{c \alpha} \| S(\tau) \|_{\Ws} \| A \| \| u_{0} \|_{\Wss}.
\end{eqnarray*}
We have used the norm-preserving properties of the reflection $U$. The estimate (\ref{Eqn.775}) was used in the second inequality and the constants $\alpha$ and $\beta$ are defined in lemma \ref{Lemma.209}. Recall that $B \psi = 0$ from (\ref{Eqn.701}). We have also used the fact that a semigroup commutes with its generator when acting on the domain of the generator. In other words, $R(t)B v = B R(t) v $ for all $v \in \Wss_{-}$. The boundedness of $S(t)$ and $R(t)$ was obtained in theorem \ref{Thm.120} as maps on $\Ws$, and in Corollary \ref{Cor.001} as maps on $\Wss_{-}$. Finally, it is clear that $A : \Wss \to \Ws$ and $B : \Wss \to \Ws$ are bounded with $\|A\| = \|B\| \leq c ( l^{-1} + \overline{\mu}_{a} + 2 \overline{\mu}_{s})$. Hence we obtain the following estimate,
\begin{eqnarray*}
\| Q u_{0} \|_{\Wss} \leq  \frac{ \left( 1 + l (\overline{\mu}_{a} + 2 \overline{\mu}_{s} ) \right)^2 }{ \alpha \beta }    
 \| R(\tau) \|_{\Ws} \| S(\tau) \|_{\Ws} \| u_{0} \|_{\Wss} .
\end{eqnarray*}

Now the stability estimate from proposition \ref{Prop.010} implies that both $\| S(\tau)\|_{\Ws}$ and $\| R(\tau) \|_{\Ws}$ decay exponentially fast as $\tau$ increases. Therefore, there exists a finite time $\tau$ such that $Q : \Wss_{-} \to \Wss_{-}$ is a contraction mapping and the desired results follow from the Neumann series theorem. 
\end{proof}

We point that theorem \ref{Thm.600} directly implies the validity of the first main result expressed as theorem \ref{Thm.MainInv}. Notice that the choice of $\psi$ in (\ref{Eqn.752}) as the solution of the reverse stationary problem (\ref{Eqn.701})-(\ref{Eqn.702}) was crucial in order to stay within the formulation of the space $\Wss$ and to obtain the estimate on the norm of $Q$. These two facts follow from the following two properties:
\begin{itemize}
\item[-] $\psi$ conforms to the boundary data $(U \Lambda u_{0})(t)$ at $t=0$, and
\item[-] $\psi$ belongs to the null space of the generator $B$.
\end{itemize}
From the well-posedness of the reversed stationary problem \ref{Def.SteadyReverseProb}, we see that the above two properties determine $\psi$  uniquely within the space $\Wss$.

To conclude this section, we prove a brief sketch for the proof of theorem \ref{Thm.MainInv2} which relies on generalizing validity of theorem \ref{Thm.600}. First of all, the time reversal operator $G$ from (\ref{Eqn.756}) can be modified to obtain a bounded operator $G : L^2([0,\tau];\Ts_{+}) \to \Ws$. This is accomplished using the concepts of mild solutions and generalized traces as in 
\cite[Section 2]{Kli-Yam-2007}, \cite[Section 14.4]{Mok-1997} or Cessenat \cite{Ces-1984,Ces-1985}. In fact, the measurement operator $\Lambda$ from (\ref{Eqn.630}) can be boundedly extended to $\Lambda : \Ws \to L^2([0,\tau];\Ts_{+})$ as shown in the proof of theorem \ref{Thm.MainControl} in Section \ref{Section:ControlProblem}. However, it would no longer make sense to speak of measured boundary data $(U \Lambda u_{0})(t)$ at time $t = 0$. Hence, the function $\psi$ and the operator $G$ have to be modified as follows. We let $v \in C([0,\tau];\Ws)$ be the mild solution of (\ref{Eqn.750})-(\ref{Eqn.754}) with $\psi \equiv 0$ and $h \in L^2([0,\tau];\Ts_{+})$. Then, following steps analogous to those in the proof of theorem \ref{Thm.600}, we can obtain an estimate on the norm of $Q : \Ws \to \Ws$,
\begin{eqnarray*}
\| Q u_{0} \|_{\Ws} \leq \| R(\tau) \|_{\Ws} \| S(\tau) \|_{\Ws} \| u_{0} \|_{\Ws}, \qquad \text{for all $u_{0} \in \Ws$}.
\end{eqnarray*}

In view of proposition \ref{Prop.010}, we see that there exists $\tau < \infty$ such that $Q : \Ws \to \Ws$ is a contraction mapping and the conclusion of theorem \ref{Thm.600} is valid with convergence in the $\Ws$-norm for initial condition $u_{0} \in \Ws$. Hence, the conclusion of theorem \ref{Thm.MainInv2} follows from the Neumann series theorem.


\section{The control problem} \label{Section:ControlProblem}

In this section, we develop the proof of our main result concerning the exact controllability problem the radiative transport. It is well-known in control theory that exact controllability in a Hilbert space setting is equivalent to the continuous observability property for the so-called adjoint problem. In turn, observability is obtained from the solvability of the inverse problem which we have already established in Section \ref{Section:InverseProblem}. The relation between these concepts is made precise using duality arguments which we proceed to describe.

We will need the following \textit{angular-reflection} operator $V : \Ws \to \Ws$ defined by
\begin{eqnarray}
(V w)(x,\theta) = w(x,-\theta), \label{Eqn.860}
\end{eqnarray}
which is clearly unitary.

\begin{proof}[Proof of theorem \ref{Thm.MainControl}]
We wish to construct the adjoint of the measurement operator $\Lambda$ defined in (\ref{Eqn.630}) viewed for now as a densely defined operator $\Lambda : \Wss_{-} \subset \Ws \to L^2([0,\tau];\Ts_{+})$. First, let $u_{0} \in \Wss_{-}$ be arbitrary and $u \in C^{1}( [0,\tau]; \Ws) \cap C( [0,\tau]; \Wss_{-})$ be the unique solution of the problem (\ref{Eqn.622})-(\ref{Eqn.624}) with $u_{0}$ as initial condition. Hence we have that $\gamma_{+} u = \Lambda u_{0}$.

Now let $h \in L^{2}([0,\tau];\Ts_{+})$ be arbitrary, and $\{ h_{k} \}_{k \geq 1} \subset C^{1}([0,\tau];\Ts_{+})$ be a sequence converging to $h$ in the $L^{2}([0,\tau];\Ts_{+})$-norm. For each $h_{k}$ let $v_{k} \in C^{1}([0,\tau]; \Ws) \cap C([0,\tau];\Wss)$ be the strong solution of the problem (\ref{Eqn.001c})-(\ref{Eqn.003c}). Hence, $v_{k}(\tau) = \Upsilon h_{k}$ where $\Upsilon : L^{2}([0,\tau];\Ts_{+}) \to \Ws$ is bounded as defined in (\ref{Eqn.010c}). 

Recall the reflector operator from (\ref{Eqn.633}) and notice that $\phi = U u$ solves
\begin{eqnarray*}
\dot{\phi}(t) = B \phi(t) \quad \text{for $0 \leq t \leq \tau$}, \\
\phi(\tau) = (Uu)(\tau),   \\
\gamma_{+} \phi(t) = 0 \quad \text{for $0 \leq t \leq \tau$}.
\end{eqnarray*}
Now integrate $0 = (\dot{v}_{k} - Av_{k}) \phi$ over the domain $[0,\tau] \times \Omega \times \Sph$ and use integration-by-parts to obtain $\la v_{k}(\tau) , \phi(\tau) \ra_{\Ws} - \la v_{k}(0) , \phi(0) \ra_{\Ws} = \la v_{k}  , \phi \ra_{L^2([0,\tau];\Ts_{-})} - \la v_{k}  , \phi \ra_{L^2([0,\tau];\Ts_{+})}$. Recall that $v_{k}(\tau) = \Upsilon h_{k}$, $v_{k}(0) = 0$, $\gamma_{-}v_{k} = h_{k}$, $\gamma_{+} \phi = 0$, and $\gamma_{-} \phi = \gamma_{-} Uu = U \gamma_{+}u = U \Lambda u_{0}$. Also notice that $\phi(\tau) = V u_{0}$. It follows that $ \la \Upsilon h_{k}  , V u_{0} \ra_{\Ws} =  \la h_{k} , U \Lambda u_{0} \ra_{L^2([0,\tau];\Ts_{-})}$. Recall that $\Upsilon$ is bounded, so after taking the limit $k \to \infty$ we obtain
\begin{eqnarray*}
\la \Upsilon h  , V u_{0} \ra_{\Ws} =  \la h  , U \Lambda u_{0} \ra_{L^2([0,\tau];\Ts_{-})}, \quad \text{for all $h \in L^{2}([0,\tau];\Ts_{+})$ and $u_{0} \in \Wss_{-}$},
\end{eqnarray*}
where $U$ and $V$ are the reflector operators defined in (\ref{Eqn.633}) and (\ref{Eqn.860}), respectively. Because $V=V^{*}=V^{-1}$ and $U=U^{*}=U^{-1}$, we find that 
\begin{eqnarray}
\Lambda^{*} = V \Upsilon U . \label{Eqn.1000}
\end{eqnarray}
Since $V$, $\Upsilon$ and $U$ are bounded operators then so is $\Lambda^{*}$ and consequently $\Lambda$ can be boundedly extended to $\Lambda : \Ws \to L^2([0,\tau];\Ts_{+})$. We have already used this extension in Section \ref{Section:InverseProblem} to obtain theorem \ref{Thm.MainInv2}. This means that $\Lambda : \Ws \to L^2([0,\tau];\Ts_{+})$ is injective and has a closed range. Because $\Upsilon^{*} = U \Lambda V$, with $U$ and $V$ being boundedly invertible, it follows from the closed range theorem that both $\rangesp{\Upsilon^{*}} \subset L^{2}([0,\tau];\Ts_{-})$ and $\rangesp{\Upsilon} \subset \Ws$ are closed.

The exact controllability problem \ref{Def.ControlProb} reduces to showing the surjectivity of the control operator $\Upsilon$. Since $\rangesp{\Upsilon}$ is closed in $\Ws$, then basic duality theory tells us that $\rangesp{\Upsilon} = \nullsp{\Upsilon^{*}}^{\perp} = \nullsp{U \Lambda V}^{\perp}$. Now, we know that $U$ and $V$ are isometries and $\Lambda$ is injective. Hence $\rangesp{\Upsilon} = \Ws$, which establishes the exact boundary controllability of the transport field. Moreover, from pseudo-inverse theory for Hilbert spaces \cite{Deu-2001}, we obtain the minimum-norm control given by $h_{\rm min} = \Upsilon^{*}(\Upsilon \Upsilon^{*})^{-1} v_{\star}$, which concludes the proof.
\end{proof}


\section{Beyond the weak scattering regime} \label{Section:BeyondWeak}

Here we address an approach to overcome the limitations imposed by the weak scattering assumption $l \overline{\mu}_{s} \, e^{l(\overline{\mu}_{a} + \overline{\mu}_{s})} < e^{-1}$ employed so far in this paper. We do this in the context of mild solutions of the transport problem for initial data in $\Ws$ and generalized boundary traces in $L^2([0,\tau];\Ts)$. 

From the proof of theorem \ref{Thm.600}, we see that we have reduced the inverse problem to the following equation
\begin{eqnarray}
(I- Q(\tau)) u_{0} = Gh, \qquad h = \Lambda u_{0}. \label{Eqn.2000}
\end{eqnarray}
Our reconstruction method is based on the decaying behavior of the operator $Q(\tau)$ as $\tau \to \infty$ which in turn is guaranteed by the weak scattering assumption $l \overline{\mu}_{s} \, e^{l(\overline{\mu}_{a} + \overline{\mu}_{s})} < e^{-1}$. However, a closer look into the proof reveals that $Q(\tau)$ can be expressed as follows,
\begin{eqnarray*}
Q(\tau) = V R(\tau) V S(\tau), 
\end{eqnarray*}
where $S$ and $R$ are the semigroups for the direct and reversed transport problems \ref{Def.DirectProb} and \ref{Def.ReverseProb}, respectively. Also, $V : \Ws \to \Ws$ is the unitary operator defined in (\ref{Eqn.860}). It turns out that the semigroup $S(\tau) : \Ws \to \Ws$ can be shown to be compact for $\tau > T$ provided that the scattering kernel $\kappa$ satisfies certain regularity in the sense of Mokhtar-Kharroubi \cite{Mokhtar-Kharroubi-2005}. Therefore, the operator $Q(\tau) : \Ws \to \Ws$ is compact for all $\tau > T$ and this compactness is independent of the size of the absorption and scattering coefficients $\mu_{\rm a}$ and  $\mu_{\rm s}$.

Therefore (\ref{Eqn.2000}) is of Fredholm type, and $u_{0} \in \Ws$ can be reconstructed in a stable manner provided that $1$ is not an eigenvalue of the compact operator $Q(\tau)$. Moreover, if this is the case, then the exact controllability theorem \ref{Thm.MainControl} holds true for $\tau > T$ with no need to assume weakly scattering media. Unfortunately, the Fredholm theory does not provide an explicit algorithm to invert the operator $(I-Q)$ nor an a-priori estimate for the constant of stability.


\ack
The author would like to thank his collaborators, Liliana Borcea and Ricardo Alonso, for fruitful discussions, and Guillaume Bal for his suggestions given during the \textit{Coupled Physics Inverse Problems} workshop at the CMM, Universidad de Chile, January 2013. The author also extends gratitude to the anonymous referees who provided the most constructive observations. This work was partially supported by the AFSOR Grant FA9550-12-1-0117, the ONR Grant N00014-12-1-0256 and by the NSF Grant DMS-0907746.


%


\section*{References}

\bibliographystyle{unsrt}
\bibliography{Biblio}

\end{document}